\documentclass{amsart}
\font\fiverm=cmr5
\font\ehsc=cmcsc10 scaled 850

\let\<=\langle
\let\>=\rangle
\let\\=\backslash
\let\sse=\subseteq
\let\noi=\noindent
\let\vphi=\varphi
\let\veps=\varepsilon
\let\limply=\Longrightarrow

\def\0{\{0\}}
\def\span{{\kern.5pt{\rm span}\kern1pt}}
\def\smallfrac#1#2{{\textstyle{\frac{#1}{#2}}}}
\def\conv{{\;\longrightarrow\;}}
\def\wconv{{{\buildrel_{\scriptstyle w}\over\conv}}}
\def\jconv{{\conv_{_{\kern-13.5pt\scriptstyle j\kern9pt}}}}
\def\kconv{{\conv_{_{\kern-13.5pt\scriptstyle k\kern9pt}}}}
\def\sslash{\hbox{{\fiverm/}}}
\def\notwconv{{{\wconv\kern-13pt\sslash}\kern9pt}}

\def\A{{\mathcal A}}
\def\B{{\mathcal B}}
\def\H{{\mathcal H}}
\def\M{{\mathcal M}}
\def\N{{\mathcal N}}
\def\Oe{{\mathcal O}}
\def\R{{\mathcal R}}
\def\X{{\mathcal X}}

\def\BX{{\B[\X]}}

\def\CC{{\mathbb C\kern.5pt}}
\def\DD{{\mathbb D\kern.5pt}}
\def\NN{{\mathbb N\kern.5pt}}
\def\RR{{\mathbb R\kern.5pt}}
\def\TT{{\mathbb T\kern.5pt}}
\def\ZZ{{\mathbb Z\kern.5pt}}

\newsymbol\varnothing 203F
\let\void=\varnothing

\def\newmatrix#1{\null\,\vcenter{
                 \baselineskip=8pt\mathsurround=-0pt\ialign{
                 \hfil ${##}$
                 \hfil &&
                 \hfil ${##}$
                 \hfil \crcr
                 \mathstrut \crcr
                 \noalign{\kern-\baselineskip}#1 \crcr
                 \mathstrut \crcr
                 \noalign{\kern-\baselineskip} \crcr }}\!}

\newtheorem{theorem}{Theorem}
\newtheorem{lemma}{Lemma}
\newtheorem{corollary}{Corollary}
\newtheorem{proposition}{Proposition}

\theoremstyle{definition}

\newtheorem{remark}{Remark}
\newtheorem{question}{Question}

\numberwithin{theorem}{section}
\numberwithin{lemma}{section}
\numberwithin{corollary}{section}
\numberwithin{proposition}{section}
\numberwithin{conjecture}{section}
\numberwithin{definition}{section}
\numberwithin{remark}{section}
\numberwithin{question}{section}

\begin{document}

\vglue-50pt\noindent
\hfill{\it Czechoslovak Mathematical Journal}\/,
{\bf 68(143)} (2018) 371--386

\vglue20pt
\title{On Weak Supercyclicity II}
\author[C.S. Kubrusly]{C.S. Kubrusly}
\address{Applied Mathematics Department, Federal University,
         Rio de Janeiro, RJ, Brazil.}
\email{carloskubrusly@gmail.com}

\author{B.P. Duggal}
\address{8 Redwood Grove, Northfield Avenue, Ealing, London W5 4SZ,
         United Kingdom}
\email{bpduggal@yahoo.co.uk}

\subjclass{Primary 47A16; Secondary 47B15}
\renewcommand{\keywordsname}{Keywords}
\keywords{Supercyclic, weakly supercyclic,
          weakly l-sequentially supercyclic operators.}
\date{July 31, 2016; revised$:$ January 7, 2017}

\begin{abstract}
This paper considers weak supercyclicity for bounded linear operators on
a normed space$.$ On the one hand, weak supercyclicity is investigated for
classes of Hilbert-space operators: (i) self-adjoint operators are not weakly
supercyclic, (ii) diagonalizable operators are not weakly l-sequentially
supercyclic, and (iii) weak l-sequential supercyclicity is preserved between
a unitary operator and its adjoint$.$ On the other hand, weak supercyclicity
is investigated for classes of normed-space operators: (iv) the point spectrum
of the normed-space adjoint of a power bounded supercyclic operator is either
empty or is a singleton in the open unit disk, (v) weak l-sequential
supercyclicity coincides with supercyclicity for compact operators, and (vi)
every compact weakly l-sequentially supercyclic operator is quasinilpotent.
\end{abstract}

\maketitle

\vskip-20pt\noi
\section{Introduction}

The reason of this paper is to characterize weak supercyclicity, in
particular, weak l-sequential supercyclicity, for bounded linear operators
on a normed space$.$ Section 2 deals with notation, terminology, and basic
notions that will be required throughout the text$.$ In Section 3 it is
shown$:$ (i) self-adjoint operators are not weakly supercyclic (Theorem 3.1),
(ii) diagonalizable operators are not weakly l-sequentially supercyclic
(Theorem 3.2), (iii) weak l-sequential supercyclicity is preserved between
a unitary operator and its adjoint (Theorem 3.3), and (iv) the point spectrum
of the normed-space adjoint of a power bounded supercyclic operator is either
empty or is a singleton in the open unit disk (Theorem 3.4), and it is also
shown when this happen for weakly l-sequentially supercyclic operators$.$ The
first result of Section 4 gives a first characterization for weakly
l-sequentially supercyclic compact operators$:$ they are supercyclic
(Theorem 4.1) --- does weak supercyclicity also coincides with weak
l-sequential supercyclicity for compact operators$\kern.5pt?$ The section
closes by giving a full spectral characterization for weakly l-sequentially
supercyclic compact operators$:$ they are quasinilpotent (Theorem 4.2).

\vskip0pt\noi
\section{Notation, Terminology and Basics}

Let $\X$ be a nonzero complex normed space and let $\X^*$ be the dual of
$\X.$ A sub\-space of $\X$ is a {\it closed}\/ linear manifold of $\X.$ If
$\M$ is a linear manifold of $\X$, then its closure $\M^-$ is a subspace$.$
The normed algebra of all operators on $\X$ (i.e., of all bounded linear
transformations of $\X$ into itself) will be denoted by $\BX.$ For any
operator $T$ on a normed space $\X$ let $\N(T)=T^{-1}\0=\{x\in\X\!:Tx=0\}$ be
the kernel of $T\kern-1pt$, which is a subspace of $\X$, and let $\R(T)=T(\X)$
be the range of $T\kern-1pt$, which is a linear manifold of $\X.$ Let $T^*\!$
in $\B[\X^*]$ stand for the normed-space adjoint of $T\!.$ We use the same
notation for the Hilbert-space adjoint of a Hilbert-space operator.

\vskip6pt
For each normed-space operator $T$ the limit
${r(T)=\lim_n\|T^n\|^\frac{1}{n}}$ exists in $\RR$ and is such that
${0\le r(T)\le\|T\|}.$ If an operator $T$ on a normed space is such that
$r(T)=0$, then it is quasinilpotent$.$ At the other end, if $T$ is such that
$r(T)=\|T\|$, then it is normaloid$.$ Let
$\sigma_{\kern-1ptP}(T)=\{{\lambda\in\CC}\!:\N({\lambda I-T})\ne\0\}$ be the
point spectrum of $T\kern-1pt$, the set of eigenvalues of $T\kern-1pt.$
An operator $T$ on a normed space $\X$ is power bounded if
${\sup_{n\ge0}\|T^n\|<\infty}$, and strongly stable or weakly stable if the
$\X$-valued sequence $\{T^nx\}_{n\ge0}\kern-1pt$ converges to zero in the
norm topology or in the weakly topology of $\X$,
$$
T^nx\conv0
\qquad\hbox{or}\qquad
T^nx\wconv0,
$$
which means ${\|T^nx\|\to0}$ or ${f(T^nx)\to0}$ for every ${f\in\X^*}\!$, for
every ${x\in\X}$, respec\-tively$.$ If $\X$ is a Banach space, so that $T$
lies in the Banach algebra $\BX$, then let ${\sigma(T)\subset\CC}$ stand for
the spectrum of $T$ (which is compact and nonempty)$.$ In this case $r(T)$
coincides with the spectral radius of $T$; that is,
$r(T)={\max_{\lambda\in\sigma(T)}|\lambda|}$ (by the Gelfand--Beurling
formula)$.$ Thus if $\X$ is a Banach space, then ${T\in\BX}$ is
quasinilpotent if and only if $\sigma(T)=\0$.

\vskip6pt
With the assumption that $\X$ is a normed space still in force, the orbit of
a vector ${y\in\X}$ under an operator ${T\in\BX}$ is the set,
$$
\Oe_T(y)={\bigcup}_{n\ge0}T^ny=\big\{T^ny\in\X\!:\,n\in\NN_0\big\},
$$
where $\NN_0$ denotes the set of nonnegative integers --- we write
${\bigcup}_{n\ge0}T^ny$ for the set
${\bigcup}_{n\ge0}T^n(\{y\})={\bigcup}_{n\ge0}\{T^ny\}.$ The orbit $\Oe_T(A)$
of a set ${A\sse\X}$ under $T$ is likewise defined$:$
$\Oe_T(A)=\bigcup_{n\ge0}T^n(A)=\bigcup_{y\in A}\Oe_T(y).$ Let $\span A$ stand
for the linear span of a set ${A\sse\X}$ and consider the projective orbit of
a vector $y$ under $T\kern-1pt$, which is the orbit of the one-dimensional
space spanned by the singleton $\{y\}$,
$$
\Oe_T(\span\{y\})={\bigcup}_{n\ge0}T^n(\span\{y\})
=\big\{\alpha T^ny\in\X\!:\,\alpha\in\CC,\,n\in\NN_0\big\}.
$$

\vskip0pt
The closure (in the norm topology of $\X$) of a set ${A\sse\X}$ is denoted by
$A^-\!$, and the weak closure (in the weak topology of $\X$) is denoted by
$A^{-w}\!.$ Thus $A$ is dense or weakly dense if $A^-\!=\X$ or $A^{-w}\!=\X$,
respectively$.$ A set $A$ is weakly sequentially closed if every $A$-valued
weakly convergent sequence has its limit in $A$, and the weak sequential
closure $A^{-sw}\!$ of $A$ is the smallest weakly sequentially closed subset
of $\X$ including $A$, and $A$ is weakly sequentially dense if $A^{-sw}\!=\X.$
The weak limit set $A^{-lw}\!$ of a set $A$ is the set of all weak limits of
weakly convergent $A$-valued sequences, and a set $A$ is weakly l-sequentially
dense if $A^{-lw}\!=\X.$ In general, the inclusions
$A^-\!\sse A^{-lw}\!\sse A^{-sw}\!\sse A^{-w}\!$ are proper (see, e.g.,
\cite[pp.38,39]{Shk}, \cite[pp.259,260]{BM2})$.$ However, if a set ${A\sse\X}$
is convex, then ${A^-\!=A^{-w}}$ (see, e.g., \cite[Theorem V.1.4]{Con} and so,
if $A$ is convex, then the above chain of inclusions become an identity.

\vskip6pt
A vector ${y\in\X}$ is supercyclic or weakly supercyclic for an operator
${T\in\BX}$ if
$$
\Oe_T(\span\{y\})^-=\X
\qquad\hbox{or}\qquad
\Oe_T(\span\{y\})^{-w}=\X,
$$
and it is weakly l-sequentially supercyclic or weakly sequentially
supercyclic if
$$
\Oe_T(\span\{y\})^{-lw}=\X
\qquad\hbox{or}\qquad
\Oe_T(\span\{y\})^{-sw}=\X,
$$
respectively$.$ An operator ${T\in\BX}$ is supercyclic, weakly l-sequentially
supercyclic, weakly sequentially supercyclic, or weakly supercyclic if it
has a supercyclic, a weakly l-sequentially supercyclic, a weakly sequentially
supercyclic, or a weakly supercyclic vector, respectively$.$ Thus
\vskip6pt\noi
$$
\newmatrix{\!
_{_{\hbox{\ehsc supercyclicity}}} & _{_{_{\textstyle\limply}}}\! &
_{\hbox{\ehsc weak l-sequential}} & _{_{_{\textstyle\limply}}}\! &
_{\hbox{\ehsc weak sequential}}   & _{_{_{\textstyle\limply}}}\! &
_{\hbox{\ehsc weak}}              & \!\!.                          \cr
                                  &                              &
\hbox{\ehsc supercyclicity}       &                              &
\hbox{\ehsc supercyclicity}       &                              &
\hbox{\ehsc supercyclicity}       &                                \cr}
$$

\vskip6pt
So a vector ${y\in\X}$ is supercyclic or weakly l-sequentially supercyclic
for an operator ${T\in\BX}$ (i.e., ${\Oe_T(\span\{y\})^-\!=\X}$ or
${\Oe_T(\span\{y\})^{-lw}\!=\X}$) if and only if for every ${x\in\X}$ there
is a $\CC$-valued sequence $\{\alpha_i\}_{i\ge0}$ (that depends on $x$ and
$y$, and consists of nonzero numbers) such that, for some subsequence
$\{T^{n_i}\}_{i\ge0}$ of $\{T^n\}_{n\ge0}$,
$$
\alpha_iT^{n_i}y\conv x
\qquad\hbox{or}\qquad
\alpha_iT^{n_i}y\wconv x,
$$
respectively$.$ Weak l-sequential supercyclicity was considered in \cite{BCS}
(and implicitly in \cite{BM}), and it was referred to as weak 1-sequential
supercyclicity in \cite{Shk}$.$ Although there are reasons for such a
terminology, we have changed it here to weak l-sequential supercyclicity,
replacing the numeral ``1'' with the letter ``l'' for ``limit''$.$ Any form
of cyclicity implies separability for $\X$ (see e.g., \cite[Section 3]{KD})$.$

\vskip6pt
The contribution to linear dynamics of the present paper in contrast with
\cite{BM, BM2,BCS,Shk} is the characterization of weak l-sequential
supercyclicity for further classes of operators (including self-adjoint,
diagonalizable, unitary, normal, hyponormal, and compact) as in
Theorems 3.1, 3.2, 3.3, 3.4 and Theorems 4.1 and 4.2$.$ These were carried
out here on Banach spaces (or normed spaces when completeness was not
necessary) or, in particular, on Hilbert spaces$.$ The stronger notion of
hypercyclicity has been investigated in Fr\'echet spaces, or F-spaces, or
locally convex spaces (see e.g., \cite{Ans,BM2,BFPW,BP,GP,Per})$.$ Some of
the above classes of operators may have a natural extention on some of these
spaces, which perhaps might be worth investigating in light of the weaker
notion of weak l-sequential supercyclicity$.$ However, we refrain from going
further than Banach spaces (or norm\-ed spaces) here to keep up with the
focus on the main topic of the paper.

\vskip0pt\noi
\section{Adjointness and Weak Supercyclicity}

The following proposition summarizes a few known results that will be often
required throughout the next two sections, which are germane to Hilbert
spaces$.$ An operator $T$ on a Hilbert space is self-adjoint or unitary if
${T^*\kern-1pt=T}$ or ${T\kern1ptT^*\kern-1pt=T^*T=I}$, respectively, where
$I$ stands for the identify operator$.$ A unitary operator is absolutely
continuous, singular-continuous, or singular-discrete if its scalar spectral
measure is absolutely continuous, singular-continuous, or singular-discrete,
respectively, with respect to normalized Lebesgue measure on the
$\sigma$-algebra of Borel subsets of the unit circle$.$ An operator is an
isometry if ${T^*T=I}$ and a coisometry if its adjoint is an isometry$.$ Thus
a unitary is an isometry and a coisometry, which means an invertible
isometry$.$ An operator is normal if ${T\kern1ptT^*\kern-1pt=T^*T}$,
hyponormal if ${T\kern1ptT^*\kern-1pt\le T^*T}$, and cohyponormal if its
adjoint is hyponormal$.$ These are all normaloid operators$.$ Some extensions
along the lines discussed in Proposition 3.1 below from hyponormal to further
classes of normaloid operators, such as paranormal operators and beyond, have
recently been considered in literature (see e.g., \cite[Corollary 3.1]{Dug},
\cite[Theorem 2.7]{DKK}), but again we refrain from going further than
hyponormal operators here to keep up with the focus on the main topic of the
paper.

\vskip4pt\noi
\begin{proposition}
The following assertions hold for Hilbert-space operators\/.
\begin{description}
\item{\rm(a)}
No hyponormal operator is supercyclic\/
$($no unitary operator is supercyclic\/$)$.
\vskip3pt
\item{\rm(b)}
A hyponormal weakly supercyclic operator is a multiple of a unitary\/.
\vskip3pt
\item{\rm(c)}
There exist weakly l-sequentially supercyclic unitary operators\/.
\vskip3pt
\item{\rm(d)}
A weakly l-sequentially supercyclic unitary operator is singular-continuous\/.
\end{description}
\end{proposition}

\begin{proof}
(a) \cite[Theorem 3.1]{Bou} (for the unitary case see also
\cite[Proof of Theorem 2.1]{AB}).
\vskip2pt
\qquad
(b) \cite[Theorem 3.4]{BM}.
\vskip2pt
\qquad
(c) \cite[Example 3.6]{BM}
(also \cite[pp.10,12]{BM}, \cite[Proposition 1.1, Theorem 1.2]{Shk}).
\vskip2pt
\qquad
(d) \cite[Theorem 4.2]{Kub}.
\end{proof}

\vskip4pt
Although a unitary operator can be weakly supercyclic, a self-adjoint cannot.

\vskip4pt\noi
\begin{theorem}
A self-adjoint operator on a Hilbert space is not weakly supercyclic.
\end{theorem}

\begin{proof}
Since a weakly supercyclic hyponormal operator is a multiple of a unitary
operator (cf$.$ Proposition 3.1(b)), if $T$ is self-adjoint on a Hilbert
space $\H$ and weakly supercyclic, then it is a self-adjoint multiple of a
unitary, which implies that $T^2$ is a positive multiple of the identity,
say, $T^2=|\beta|^2 I$ and so $T^n=|\beta|^nI$ if $n$ is even or
$T^n=|\beta|^{n-1}T$ if $n$ is odd$.$ Thus the projective orbit of any
vector ${z\in\H}$ is included in a pair of one-dimensional subspaces,
\vskip2pt\noi
\begin{eqnarray*}
\Oe_T(\span\{z\})
&\kern-6pt=\kern-6pt&
\big\{\alpha T^nz\in\H\!:\;\alpha\in\CC,\;n\in\NN_0\big\}             \\
&\kern-6pt\sse\kern-6pt&
\big\{\alpha z\in\H\!:\;\alpha\in\CC\big\}
\cup\big\{\alpha Tz\in\H\!:\;\alpha\in\CC\big\}
=\span\{z\}\cup\span\{Tz\},
\end{eqnarray*}
\vskip2pt\noi
which is not dense in $\H$ in the weak topology if ${\dim\H>1}.$ Hence a
self-adjoint operator $T$ (on a space of dimension greater than 1) is not
weakly supercyclic.
\end{proof}

\vskip3pt
Normal operators are not supercyclic (hyponormal are not) but can
be weakly l-sequentially supercyclic (unitary can), but diagonalizable
operators cannot.

\vskip4pt\noi
\begin{theorem}
A diagonalizable operator on a Hilbert space is not weakly l-se\-quentially
supercyclic\/.
\end{theorem}

\begin{proof}
A diagonalizable operator $T$ on a Hilbert space $\H$ is precisely an
operator unitarily equivalent to a diagonal operator (see, e.g.,
\cite[Proposition 3.A]{ST})$.$ So it is normal and therefore if it is
weakly supercyclic, then it acts on a separable Hilbert space (i.e., $\H$
is separable), and it is a multiple of a unitary operator (cf$.$
Proposition 3.1(b))$.$ Thus such a unitary operator is unitarily equivalent
to a unitary diagonal $U\!$ on $\ell_+^2$, which is discrete (i.e.,
singular-discrete)$.$ If, in addition, $T$ is weakly l-sequentially
supercyclic, then so is $U\!$, and hence $U\!$ must be singular-continuous
(cf$.$ Proposition 3.1(d)), which is a contradiction$.$ Then a
diagonalizable operator is not weakly l-sequentially supercyclic.
\end{proof}

\vskip0pt\noi
\begin{remark}
It was asked in \cite[Question 5.1]{Kub} {\it whether every weakly
supercyclic unitary operator is singular-continuous}\/$.$ An affirmative
answer ensures that Theorem 3.2 holds if weakly l-sequentially supercyclic
is replaced by weakly supercyclic.
\end{remark}

\vskip3pt
If $T$ is an invertible supercyclic, then so is its inverse
\cite[Section 4]{AB}, \cite[Corollary 2.4]{Sal}$.$ There are, however,
invertible weakly supercyclic operators in $\B[\ell^p]=\B[\ell^p(\ZZ)]$ for
any ${p\in[2,\infty)}$ whose inverses are not weakly supercyclic
\cite[Corollary 2.5]{San1}$.$ For ${p=2}$ this exhibits a Hilbert-space
invertible operator whose inverse is not weakly supercyclic$.$ Since for
${p=2}$ such an operator is not unitary, the following question crops up$:$
if a unitary operator is weakly supercyclic, is its inverse (i.e., its
adjoint) weakly supercyclic$?$ Recall that the adjoint of a supercyclic
coisometry may not be supercyclic (example$:$ a backward unilateral shift
$S^*$ is a supercyclic coisometry \cite[Theorem 3]{HW}, while its adjoint, the
unilateral shift $S$, being an isometry is not supercyclic \cite[p.564]{HW}
(also see \cite[Proof of Theorem 2.1]{AB}, \cite[Lemma 4.1(b)]{KD})$.$ The
same happens with weak supercyclicity$:$ the adjoint of a weakly supercyclic
coisometry may not be weakly supercyclic (example$:$ $S$ is not weakly
supercyclic by Proposition 3.1(b), but $S^*$ is weakly supercyclic, since it
is supercyclic)$.$ However, if an isometry is invertible and weakly
l-sequentially supercyclic, then it has a weakly l-sequentially supercyclic
adjoint (i.e., inverse), as we show in Theorem 3.3 below.

\vskip6pt
Let $\DD$ stand for the open unit disk (about the origin in the complex
plane $\CC$), let $\DD^-\!$ (the closure of $\DD$) stand for the closed unit
disk, and let their boundary $\TT=\partial\DD$ stand for unit circle (about
the origin).

\vskip4pt\noi
\begin{theorem}
A unitary operator on a Hilbert space is weakly l-sequentially supercyclic
if and only if its adjoint is weakly l-sequentially supercyclic\/.
\end{theorem}

\begin{proof}
We split the proof into 2 parts.

\vskip6pt\noi
{\bf Part 1.}
$\!$Let $\mu$ be a positive measure on the $\sigma$-algebra $\A_\TT$ of Borel
subsets of the unit circle $\TT$ and consider the Hilbert space
$L^2({\TT,\mu}).$ Let ${\vphi\!:\TT\to\TT}$ denote the identity function,
${\vphi(\gamma)=\gamma}$ $\mu$-a.e$.$ for ${\gamma\in\TT}$, and consider the
multiplication opera\-tor ${U_{\!\mu}\!:\!L^2(\TT,\mu)\!\to\!L^2(\TT,\mu)}$
induced by $\vphi$, $\,{U_{\!\mu}\,\psi=\vphi\,\psi}$, which is given by
$$
(U_{\!\mu}\,\psi)(\gamma)=\vphi(\gamma)\,\psi(\gamma)=\gamma\,\psi(\gamma)
\quad\hbox{$\mu$-a.e$.$ for}\;\;\gamma\in\TT,
$$
so that ${U^*_{\!\mu}\,\psi=\overline\vphi\,\psi}$, which is given by
$$
(U^*_{\!\mu}\,\psi)(\gamma)=\overline\vphi(\gamma)\,\psi(\gamma)
=\overline\gamma\,\psi(\gamma)
\quad\hbox{$\mu$-a.e$.$ for} \;\;\gamma\in\TT,
$$
for every ${\psi\in L^2(\TT,\mu)}.$ It is clear that $U_{\!\mu}$ is
unitary$.$ Let ${C\!:\!L^2(\TT,\mu)\to L^2(\TT,\mu)}$ denote the complex
conjugate transformation (i.e., $C(\psi)=\overline\psi)$, which has the
following properties$:$ it is a contraction (thus norm continuous), weakly
continuous (in fact,
${\<C(\zeta_k-\zeta)\,;\psi\>}=\overline{\<\zeta_k-\zeta\,;\overline\psi\>}$
for every ${\zeta_k,\zeta,\psi\in L^2(\TT,\mu)}\kern1pt)$, an involution
(i.e., ${C^2=I})$, additive, and conjugate homogeneous (i.e.,
${C(\alpha f)=\overline\alpha\,C(f)}$)$.$

\vskip6pt\noi
{\it Claim $1$}\/.
\qquad $C\,U_{\!\mu}=U^*_{\!\mu}C$.

\vskip2pt\noi
{\it Proof}\/.
$\!C(U_{\!\mu}\,\psi)
=\overline{(U_{\!\mu}\,\psi)}
=\overline{\vphi\,\psi}
=\overline\vphi\overline{\,\psi}=\overline \vphi\,C\psi)
=U^*_{\!\mu}(C\psi)$
for any ${\psi\in L^2(\TT,\mu)}.\!\!\!\qed$
\vskip6pt\noi
{\it Claim $2$}\/.
Let $\{U^{n_k}\}_{k\ge0}$ be an arbitrary subsequence of $\{U^n\}_{n\ge0}$,
let $\{\alpha_k\}_{k\ge0}$ be any sequence of scalars, and let ${\phi,\psi}$
be functions in $L^2(\TT,\mu).$ Then
$$
\alpha_k\,U^{n_k}_{\!\mu}\phi\wconv\psi
\quad\;\hbox{if and only if}\;\quad
\overline\alpha_kU^{*n_k}_{\!\mu}\overline\phi\wconv\overline\psi.
$$

\vskip2pt\noi
{\it Proof}\/.
Since $C$ is weakly continuous, it follows by Claim 1 (since $C$ is conjugate
homogeneous) that if ${\alpha_k\,U^{n_k}_{\!\mu}\phi\wconv\psi}$,
then
$$
\overline\alpha_k U^{*n_k}_{\!\mu}\overline\phi
=\overline\alpha_k U^{*n_k}_{\!\mu}(C\phi)
=\overline\alpha_kC(U^{n_k}_{\!\mu}\,\phi)
=C(\alpha_k\,U^{n_k}_{\!\mu}\phi)\wconv C(\psi)
=\overline\psi.
$$
Dually, since $C$ and the adjoint operation are involutions the converse
holds$.\!\!\!\qed$

\vskip6pt\noi
Take an arbitrary ${\psi\in L^2(\TT,\mu)}.$ If $U_{\!\mu}$ is weakly
l-sequentially supercyclic, then there is a supercyclic vector
${\phi\in L^2(\TT,\mu)}$ for $U_{\!\mu}$, a sequence of nonzero numbers
$\{\alpha_k(\phi,\psi)\}_{k\ge0}$, and a corresponding subsequence
$\{U^{n_k}_{\!\mu}\}_{k\ge0}$ of $\{U^n_{\!\mu}\}_{n\ge0}$ such that
$$
\alpha_k(\phi,\psi)\,U^{n_k}_{\!\mu}\phi\wconv\overline\psi.
$$
According to Claim 2 this happens if and only if
$$
\overline\alpha_k(\phi,\psi)\,U^{*n_k}_{\!\mu}\overline\phi
\wconv\overline{\overline\psi}=\psi,
$$
and so $\overline\phi$ is a weakly l-sequentially supercyclic vector for
$U^*_{\!\mu}$, and hence $U^*_{\!\mu}$ is weakly l-sequentially
supercyclic$.$ Again the converse holds dually$.$ Therefore,
\vskip6pt\noi
\centerline{\it $U_{\!\mu}$ is weakly l-sequentially supercyclic if and only
if its adjoint $U^*_{\!\mu}$ is\/.}
\vskip2pt\noi

\vskip6pt\noi
{\bf Part 2.}
Take a unitary operator $U\!$ on a Hilbert space $\H.$ If $U\!$ is weakly
supercyclic, then it is weakly cyclic, and so it is cyclic (i.e., if there
exists a vector ${y\in\H}$ such that $\Oe_U(\span\{y\})^{-w}\!=\H$, then
$\big(\span\,{\bigcup}_nU^ny\big)^{_{\scriptstyle-}}\!\!
=\!\big(\span\Oe_U(y)\big)^{_{\scriptstyle-}}\!\!
=\!\big(\span\Oe_U(y)\big)^{_{\scriptstyle-w}}\!=\H$
because $\span\Oe_U(y)$ is convex)$.$ Cyclicity implies star-cyclicity, which
in turn implies separability for $\H$ --- since $U$ is normal, star-cyclicity
for $U\!$ means$:$ there exists a vector ${y\in\H}$ for which
$\big(\span\,{\bigcup}_nU^nU^{*n}y\big)^{_-}\!\!=\H$ --- see, e.g.,
\cite[pp.73,74]{ST}$.$ Star-cyclicity ensures, by the Spectral Theorem, that
$U\!$ is unitarily equivalent to a unitary multiplication operator $U_{\!\mu}$
on $L^2(\TT,\mu)$ induced by the identity function ${\vphi\!:\TT\to\TT}$ (thus
of multiplicity one), where the positive measure $\mu$ on $\A_\TT$ is finite
and supported on ${\sigma(U)\sse\TT}$ --- see, e.g., \cite[part (a), proof of
Theorem 3.11]{ST})$.$ If, in addition, the unitary $U\!$ on $\H$ is weakly
l-sequentially supercyclic, then so is the unitary multiplication operator
$U_{\!\mu}$ on $L^2(\TT,\mu)$ (which is unitarily equivalent to it), and the
result of Part 1 ensures that $U^*_{\!\mu}$ is weakly l-se\-quentially
supercyclic, and so is the unitary $U^*\!$ (which again is unitarily
equivalent to $U^*_{\!\mu}).$ Dually, if $U^*\!$ is weakly l-sequentially
supercyclic, then so is $U\!$.
\end{proof}

\vskip0pt\noi
\begin{corollary}
A hyponormal\/ $($normal\/$)$ operator is weakly l-sequentially supercyclic
if and only if its adjoint is weakly l-sequentially supercyclic.
\end{corollary}

\begin{proof}
Proposition 3.1(b) and Theorem 3.3.
\end{proof}

\vskip4pt
If $T$ is a power bounded operator on a Banach space, then ${r(T)\le1}$
(equivalently, ${\sigma(T)\sse\DD^-}$) and so
${\sigma_{\kern-1ptP}(T)\sse\DD^-}\!.$ As we will see in the proof
Theorem 4,2, if an operator $T$ on a Banach space is supercyclic, then the
point spectrum of its normed-space adjoint $\sigma_{\kern-1ptP}(T^*)$ has at
most one element$.$ As a consequence of the forthcoming \hbox{Theorem} 3.4,
if a supercyclic operator $T$ is power bounded, then this possible unique
element $\lambda$ of $\sigma_{\kern-1ptP}(T^*)$ is such that ${|\lambda|<1}$
(so that ${\sigma_{\kern-1ptP}(T^*)\sse\{\lambda\}\subset\DD}$).

\vskip6pt
To proceed we need the following definition$.$ A normed space $\X$ is said to
be of {\it type 1}\/ if convergence in the norm topology for an arbitrary
$\X$-valued sequence $\{x_k\}$ coincides with weak convergence plus
convergence of the norm sequence $\{\|x_k\|\}$ (i.e.,
${x_k\conv x}$ $\iff$ $\big\{{x_k\wconv x}$ and ${\|x_k\|\to\|x\|}\big\}$ ---
also called {\it Radon--Riesz space}\/ and the {\it Radon--Riesz property}\/,
respectively; see, e.g., \cite[Definition 2.5.26]{Meg})$.$ Hilbert spaces
are Banach spaces of type 1 \cite[Problem 20]{Hal}.

\vskip4pt\noi
\begin{theorem}
Let\/ ${T\in\BX}$ be a supercyclic\/ $($or weakly l-sequentially
supercyclic\/$)$ operator on a normed space\/ $\X.$ Suppose there exists a
nonzero eigenvalue\/ $\lambda$ of\/ $T^*$ {\rm (i.e.,}
${0\ne\lambda\in\sigma_{\kern-1ptP}(T^*)})$ and take any nonzero eigenvector\/
${f\kern-1pt\in\X^*\!}$ of\/ ${T^*\kern-1pt\in\B[\X^*]}$ associat\-ed with\/
$\lambda$ {\rm (i.e.,} ${0\ne f\kern-1pt\in\N(\lambda I-T^*))}.$ Then for
every supercyclic\/ $($or weakly l-sequentially supercyclic\/$)$ vector\/ $y$
for $T\kern-1pt$ and every ${x\in\X}$ such that ${f(x)\ne0}$ {\rm (i.e.,}
${x\in\X\\\N(f)})$, there exists a subsequence\/ $\{n_k\}_{k\ge0}$ of
the integers\/ $\{n\}_{n\ge0}$ such that
$$
\smallfrac{f(x)}{f(y)}\smallfrac{1}{\lambda^{n_k}}T^{n_k}y\conv x
\qquad({\rm or}\quad
\smallfrac{f(x)}{f(y)}\smallfrac{1}{\lambda^{n_k}}T^{n_k}y\wconv x)
$$
$(${\rm i.e., } we may set\/
${\alpha_k(x,y)=\smallfrac{f(x)}{f(y)}\smallfrac{1}{\lambda^{n_k}}}).$
In particular,
$$
\smallfrac{1}{\lambda^{n_k}}T^{n_k}y\conv y
\qquad({\rm or}\quad
\smallfrac{1}{\lambda^{n_k}}T^{n_k}y\wconv y).
$$
\vskip-2pt\noi
Moreover,
\vskip2pt\noi
\begin{description}
\item{$\kern-4pt$\rm(a)$\kern2pt$}
If\/ $T$ is power bounded and supercyclic, then\/ ${|\lambda|<1}$.
\vskip2pt
\item{$\kern-4pt$\rm(b)}
If\/ $T$ is power bounded and weakly l-sequentially supercyclic on a type 1
normed space\/, and if\/ $|f(y)|=\|f\|\lim_n\|T^ny\|$ for some weakly
l-sequentially supercyclic vector\/ $y$ and some\/
${0\ne f\kern-1pt\in\N(\lambda I-T^*)}$,then\/ ${|\lambda|<1}$.
\end{description}
\end{theorem}

\begin{proof}
Let ${T\in\BX}$ be an operator on a normed space $\X$ and consider its
normed-space adjoint ${T^*\in\B[\X^*]}.$ Let $\lambda$ be a nonzero
eigenvalue of $T^*$ and take any nonzero eigenvector ${f\in\X^*}$ associated
with the nonzero eigenvalue $\lambda$ of $T^*$ so that
$$
f(T^nx)=\lambda^n f(x)                                              \leqno(*)
$$
for every ${n\ge0}$ and every ${x\in\X}.$ (Indeed,
$f(T^nx)=(T^{n*}f)(x)=(T^{*n}f)(x)=(\lambda^n f)(x)=\lambda^n f(x).)$ Suppose
$T$ is supercyclic (or weakly l-sequentially supercyclic)$.$ Fix an arbitrary
(nonzero) supercyclic $($or weakly l-sequentially supercyclic\/$)$ vector
${y\in\X}$ for $T\kern-1pt$, and take an arbitrary vector ${x\in\X}.$ Thus
there is a sequence of nonzero numbers $\{\alpha_k(y,x)\}_{k\ge0}$ and a
corresponding subsequence $\{T^{n_k}\}_{k\ge0}$ of $\{T^n\}_{n\ge0}$ (which
depends on $x$ and $y$) such that
$$
\alpha_k(y,x)\,T^{n_k}y\conv x
\qquad({\rm or}\quad
\alpha_k(y,x)\,T^{n_k}y\wconv x).
$$
So, according to ($*$) --- for the supercyclic case recall that $f$ is
continuous,
$$
\alpha_k(y,x)\lambda^{n_k}f(y)=\alpha_k(y,x)f(T^{n_k}y)
=f(\alpha_k(y,x)\,T^{n_k}y)\to f(x).
$$
Observe that
$$
f(y)\ne0.                                                          \leqno(**)
$$
(Indeed, by the above convergence if ${f(y)=0}$, then ${f(x)=0}$
for every ${x\in\X}$, which is a contradiction)$.$ Hence,
$$
\alpha_k(y,x)\lambda^{n_k}\to\smallfrac{f(x)}{f(y)},
$$
and so if ${f(x)\ne0}$ (which ensures
$\frac{1}{\alpha_k(y,x)}\frac{1}{\lambda^{n_k}}\to\smallfrac{f(y)}{f(x)}$),
then
$$
\smallfrac{f(x)}{f(y)}\smallfrac{1}{\lambda^{n_k}}T^{n_k}y\conv x
\qquad({\rm or}\quad
\smallfrac{f(x)}{f(y)}\smallfrac{1}{\lambda^{n_k}}T^{n_k}y\wconv x).
$$
In particular, by setting ${x=y}$,
$$
\smallfrac{1}{\lambda^{n_k}}T^{n_k}y\conv y
\qquad({\rm or}\quad
\smallfrac{1}{\lambda^{n_k}}T^{n_k}y\wconv y).
$$
\vskip-2pt\noi
Moreover,
\vskip6pt\noi
(a) {\it if a power bounded operator\/ $T$ on a normed space\/ $\X$ is
supercyclic, then it is strongly stable}\/ \cite[Theorem 2.2]{AB}$.$ Thus
in this case ${T^{n_k}y\conv 0}$ so that ${\lambda^{n_k}\to0}$ (since
${\smallfrac{1}{\lambda^{n_k}}T^{n_k}y\conv y\ne0}$), which implies
${|\lambda|<1}$.

\vskip6pt\noi
(b) Suppose ${|\lambda|=1}.$ Form ($*$) and ($**$),
$$
|f(T^ny)|=|f(y)|\ne0
$$
\vskip-4pt\noi
for every ${n\ge0}$, so that
$$
f(T^ny)\not\to0
\quad\;\hbox{and}\;\quad
0<\liminf_n|f(T^ny)|,
$$
for every weakly l-sequentially supercyclic vector $y$ and every nonzero
eigenvector $f$ associated with the eigenvalue $\lambda.$ Then
${T^ny\notwconv0}$ so that $T$ is not weakly stable$.$ However, according to
\cite[Theorem 6.2]{KD} {\it if a power bounded operator\/ $T$ on a type 1
normed space\/ $\X$ is weakly l-sequentially supercyclic, then either
\begin{description}
\item{$\kern-9pt$\rm(i)$\kern6pt$}
$T$ is weakly stable, $\;\;$ or
\vskip2pt
\item{$\kern-11pt$\rm(ii)$\kern5pt$}
if\/ $y$ is any weakly l-sequentially supercyclic vector such that\/
${T^ny\notwconv0}$, then for every nonzero\/ ${f\in\X^*}$ such that\/
${f(T^ny)\not\to0}$ either
\end{description}
$\liminf_n|f(T^ny)|=0,\;\;$ or
\vskip2pt\noi
$\limsup_j|f(T^{n_k}y)|<\|f\|\kern-1pt\limsup_j\!\|T^{n_j}y\|$ for some
subsequence\/ $\{n_j\}_{j\ge0}$ of\/ $\{n\}_{n\ge0}$.}
\vskip6pt\noi
Consider a weakly l-sequentially supercyclic power bounded operator $T$ on a
type 1 normed space$.$ By the above results if ${|\lambda|=1}$, then
$|f(y)|<\|f\|\limsup_{n_j}\!\|T^{n_j}y\|$ for some subsequence
$\{n_j\}_{j\ge0}$ of $\{n\}_{n\ge0}$, for every weakly l-sequentially
supercyclic vector $y$ and every nonzero eigenvector $f$ associated with the
eigenvalue $\lambda$ (since $|f(T^ny)|=|f(y)|).$ Therefore, if
$|f(y)|=\|f\|\lim_n\|T^ny\|$ for some weakly l-sequentially supercyclic vector
$y$ and some nonzero eigenvector $f$ associated with the eigenvalue $\lambda$,
then ${|\lambda|\ne1}$, and hence ${|\lambda|<1}$ (since $T$ is power bounded).
\end{proof}

\vskip0pt\noi
\begin{remark}
Since isometries are weakly supercyclic only if they are unitaries (cf$.$
Proposition 3.1(b)), and since there exist weakly l-sequentially supercyclic
unitaries (cf$.$ Proposition 3.1(c)), it follows by Theorem 3.4 that if $T$
is a weakly l-sequentially supercyclic isometry on a Hilbert space (so that
it is unitary, and so is its adjoint), then ${|\lambda|<1}$ whenever
${\lambda\in\sigma_{\kern-1ptP}(T^*)}$ so that
$\sigma_{\kern-1ptP}(T^*)=\void.$ Actually, by Theorem 3.3 the unitary $T^*$
is a weakly l-sequentially supercyclic as well, and Proposition 3.1(d) says
that the weakly l-sequentially supercyclic unitaries $T$ and $T^*$ must be
singular-continuous, and {\it cyclic}\/ (in particular, weakly l-sequentially
supercyclic) {\it singular-continuous unitaries have no eigenvalues}\/$.$
[\,Indeed, if a star-cyclic (equivalently, a cyclic) singular-continuous
unitary has an eigenvalue, then there exists a unitary multiplication operator
$U_{\!\mu}$, induced by the identity function, with respect to some positive
singular-continuous measure $\mu$ on $\A_\TT$ (after the Spectral Theorem),
which has an eigenvalue $\lambda$ and this implies
${\gamma\,\psi(\gamma)}={\lambda\,\psi(\gamma)}$ $\,\mu$-a.e$.$ for
${\gamma\in\TT}$ for some nonzero eigenvector $\psi$ associated with the
eigenvalue $\lambda.$ So ${\gamma=\lambda}$ for every
${\gamma\in\TT\\\N(\psi)\in\A_\TT}.$ Therefore, since
${\mu(\TT\\\N(\psi))\ne0}$, we get ${\mu(\{\lambda\})>0}$ which is
a contradiction (because, being continuous, $\mu$ is null when acting on
measurable singletons)\,].
\end{remark}

\vskip0pt\noi
\section{Compactness and Weak Supercyclicity}

To begin with we need an auxiliary result on the range $\R(T)$ of an operator
$T$.

\vskip4pt\noi
\begin{lemma}
If\/ an operator\/ $T$ on a normed space\/ $\X$ is weakly supercyclic\/,
then
$$
\R(T)^-\!=\R(T)^{-wl}\!=\R(T)^{-w}=\X.
$$
\end{lemma}

\begin{proof}
If a set ${A\sse\X}$ is convex, then ${A^-\!=A^{-w\!}}$, and
so $A^-\!=A^{-wl}\!=A^{-w}\!.$ Since a linear manifold is trivially convex,
$$
\R(T)^-\!=\R(T)^{-wl}\!=\R(T)^{-w}.
$$
But the projective orbit of any vector ${u\in\X}$ is included in
${\R(T)\cup\span\{u\}}$,
\begin{eqnarray*}
\Oe_T(\span\{u\})
&\kern-6pt=\kern-6pt&
\big\{\alpha T^nu\in\X\!:\;\alpha\in\CC,\;n\in\NN_0\big\}                 \\
&\kern-6pt\sse\kern-6pt&
\span\{u\}\cup\{z\in\X\!:\,z=Tx\;\hbox{for some}\;x\in\X\}
=\span\{u\}\cup\R(T).
\end{eqnarray*}
\vskip2pt\noi
Thus if $T$ is weakly supercyclic, then $\Oe_T(\span\{y\})^{-w}=\X$
for some ${y\in\X}$, so that ${\R(T)^{-w}=\X}$.
\end{proof}

\vskip4pt
Theorem 4.1 gives a first characterization for weakly l-sequentially
supercyclic compact operators$:$ they are supercyclic.

\vskip0pt\noi
\begin{theorem}
A compact operator on a normed space is weakly l-sequentially supercyclic
if and only if it is supercyclic.
\end{theorem}

\begin{proof}
Suppose $T$ is weakly l-sequentially supercyclic operator on a normed space
$\X.$ Take an arbitrary ${x\in\X}$ so that ${x\in\R(T)^-}$ according to
Lemma 4.1$.$ Thus there exists an $\X$-valued sequence $\{x_k\}_{k\ge0}$ such
that
$$
Tx_k\conv x.
$$
Since $T$ is weakly l-sequentially supercyclic, there exists a nonzero
vector ${y\in\X}$ such that for each $x_k$ there exists a sequence of
nonzero numbers $\{\alpha_j(x_k)\}_{j\ge0}$ and a corresponding subsequence
$\{T^{n_j}\}_{j\ge0}$ of $\{T^n\}_{n\ge0}$ such that
$$
\alpha_j(x_k)\,T^{n_j}y\wconv x_k.
$$
If in addition $T$ is compact, then
$$
\alpha_j(x_k)\,T^{n_j+1}y\conv Tx_k
$$
for every $k$ (convergence in the norm topology --- see e.g.,
\cite[Problem 4.69]{EOT})$.$ Thus
$$
\alpha_j(x_k)\,T^{n_j+1}y\jconv Tx_k\kconv x.                    \leqno(*)
$$
This ensures the existence of a sequence of nonzero numbers
$\{\alpha_i(x)\}_{i\ge0}$ such that
$$
\alpha_i(x)\,T^{n_i}y\conv x                                     \leqno(**)
$$
for some subsequence $\{T^{n_i}\}_{i\ge0}$ of $\{T^n\}_{n\ge0}.$ Indeed,
consider both convergences in ($*$)$.$ Take an arbitrary ${\veps>0}.$ Thus
there exists a positive integer $k_\veps$ such that
${\|Tx_k-x\|\le\frac{\veps}{2}}$ whenever ${k\ge k_\veps}.$ Moreover, for
each $k$ there exist a positive integer $j_{\kern1pt\veps,k}$ such that
${\|\alpha_j(x_k)\,T^{n_j+1}y-Tx_k\|\le\frac{\veps}{2}}$
whenever ${j\ge j_{\veps,k}}.$ Therefore,
\vskip4pt\noi
$$
j\ge j_{\veps,k_\veps}
\quad\limply\quad
\|\alpha_j(x_{k_\veps})\,T^{n_j+1}y-x\|
\le\|\alpha_j(x_{k_\veps})\,T^{n_j+1}y-Tx_{k_\veps}\|+\|Tx_{k_\veps}\!\!-x\|
\le\veps.
$$
\vskip4pt\noi
For each integer ${i\ge1}$ set ${\veps=\frac{1}{i}}.$ Consequently, set
${k(i)=k_\veps=k_{\frac{1}{i}}}$ and
$j(i)=j_{\veps,k_\veps}\!=j_{\frac{1}{i},k_i}$, so that
$\alpha_j(x_{k_\veps})=\alpha_j(x_{k(i)}).$ Thus for every integer
${i\ge1}$ there is another integer ${j(i)\ge1}$ such that
${\|\alpha_j(x_{k(i)})\,T^{n_j+1}y-x\|}\le\smallfrac{1}{i}$ whenever
${j\ge j(i)}.$ Hence,
$$
\|\alpha_{j(i)}(x_{k(i)})\,T^{n_{j(i)}+1}y-x\|\le\smallfrac{1}{i}
\quad\;\hbox{for every integer ${i\ge0}$},
$$
and so there exists a sequence of nonzero numbers
$\{\alpha_{j(i)}(x_{k(i)})\}_{i\ge0}$ for which
$$
\alpha_{j(i)}(x_{k(i)})\,T^{n_{j(i)}+1}y\conv x.
$$
By setting ${\alpha_i(x)=\alpha_{j(i)}(x_{k(i)})}$ and
$T^{n_i}=T^{n_{j(i)}+1}$ we get$:$ there exists a sequence of nonzero numbers
$\{\alpha_i(x)\}_{i\ge0}$ and a subsequence $\{T^{n_i}\}_{i\ge0}$ of
$\{T^n\}_{n\ge0}$ such that ($**$) holds true$.$ Thus $T$ is supercyclic
(since $x$ was taken to be an arbitrary vector in $\X$)$.$ Therefore
if $T$ is weakly l-sequentially supercyclic, then $T$ is supercyclic$.$ The
converse is trivial.
\end{proof}

\vskip4pt
The next result gives an elementary proof that the classical Volterra operator
${V\!\in\B[L^p[0,1]]}$, given by $V(f)(s)=\int_0^sf(t)\,dt$ for every
${f\in L^p[0,1]}$ for ${p\ge1}$, is not weakly l-sequentially supercyclic$.$
A previous nonelementary proof that the Volterra operator is not even weakly
supercyclic was given in \cite[Section 2]{MS}.

\vskip4pt\noi
\begin{corollary}
The Volterra operator is not weakly l-sequentially supercyclic.
\end{corollary}

\begin{proof}
It was show in \cite{GM} that the Volterra operator is not supercyclic$.$
It is well known that the Volterra operator is compact (see, e.g.,
\cite[Example 7.8]{AA})$.$ Thus the Volterra operator is not weakly
l-sequentially supercyclic by Theorem 4.1.
\end{proof}

\vskip0pt\noi
\begin{question}
Does weak supercyclicity coincides with weak l-sequential supercyclicity for
compact operators on normed spaces$\kern.5pt?$
\end{question}

\vskip4pt
Theorem 4.1 yields an immediate proof that a compact hyponormal (equivalently,
a compact normal) operator is not weakly l-sequentially supercyclic$.$

\vskip4pt\noi
\begin{corollary}
A compact hyponormal is not weakly l-sequentially supercyclic.
\end{corollary}

\begin{proof}
A hyponormal operator on a Hilbert space is never supercyclic (cf$.$
Proposition 3.1(a))$.$ Thus the claimed result follows from Theorem 4.1.
\end{proof}

\vskip0pt\noi
\begin{remark}
The above result can be proved without using Theorem 4.1 neither
Proposition 3.1(a), but using the results in Proposition 3.1(b,d) as
follows$.$ Suppose a nonzero operator $T$ on a Hilbert space is weakly
supercyclic$.$ If $T$ is compact and hyponormal, then it is a compact nonzero
multiple of a unitary $U\!$, since a weakly supercyclic hyponormal is a
multiple of a unitary (cf$.$ Proposition 3.1(b))$.$ Thus $T$ and so $U$ are
invertible compact, and hence they must act on a finite-dimensional space
(since the collection of all compact operators on a normed space $\X$ is an
ideal of $\BX$, and the identity operator is not compact on an
infinite-dimensional space)$.$ On the other hand, a weakly l-sequentially
supercyclic unitary operator is singular-continuous (cf$.$
Proposition 3.1(d)), and so it must act on an infinite-dimensional space
(since on a finite-dimensional space spectra are finite, where unitaries are
singular-discrete)$.$ This leads to a contradiction$.$ Thus a compact
hyponormal operator is not weakly l-sequentially supercyclic.
\end{remark}

\vskip4pt
As it is well-known, a compact operator is hyponormal if and only if it is
compact and normal (see, e.g., \cite[Problem 6.23]{EOT}), which means a
compact diagonalizable; equivalently, a countable weighted sum of projections
(after the Spectral Theorem)$.$ Thus the result in Corollary 4.2 (compact
hyponormal are not weakly l-sequentially supercyclic; and so not supercyclic)
also follows from and Theorem 3.2.

\vskip6pt
Theorem 4.2 fully characterizes weakly l-sequentially supercyclic 
compact opera\-tors$:$ they are quasinilpotent.

\vskip4pt\noi
\begin{theorem}
A compact weakly l-sequentially supercyclic operator is quasinilpotent\/
$($acting on an infinite-dimensional Banach space\/$)$.
\end{theorem}

\begin{proof}
Take ${T\in\BX}$, where $\X$ is a normed space, and let
${T^{m*}\!\in\B[\X^*]}$ be the normed-space adjoint of ${T^m\!\in\BX}$ for an
arbitrary nonnegative integer $m.$ Suppose $T$ is compact and weakly
l-sequentially supercyclic (thus weakly supercyclic)$.$ Let
$\sigma_{\kern-1ptP}(T^*)$ be the point spectrum of $T^*\!.$ Theorem 4.1 says
that the compact $T$ is super\-cyclic, and hence
${\#\sigma_{\kern-1ptP}(T^*)\le1}$; that is, $T^*$ has at most one
eigenvalue$.$ (This has been verified for supercyclic operators in a Hilbert
space setting in \cite[Proposition 3.1]{Her}, and extended to a normed space
setting in \cite[Theorem 3.2]{AB})$.$ Since the operator $T$ in $\BX$ is
compact, its normed-space adjoint $T^*$ in $\B[\X^*]$ is compact as well (see,
e.g., \cite[Theorem 4.15]{Sch})$.$ The dual $\X^*$ of a normed space $\X$ is
a Banach space, and so the spectrum of $T^*$ is nonempty$.$ Since $T^*$ is
compact, ${\sigma(T^*)\\\0}={\sigma_{\kern-1ptP}(T^*)\\\0}$ (Fredholm
alternative)$.$ Moreover, if $\X$ is infinite dimensional, then so is
$\X^*\!$, and hence ${0\in\sigma(T^*)}$ (i.e., zero lies in $\sigma(T^*)$
since an invertible compact operator must act on a finite-dimensional
space)$.$ Summing up$:$ if $T$ is a weakly l-sententially supercyclic
compact operator on a infinite-dimensional normed space, then
$$
\#\sigma_{\kern-1ptP}(T^*)\le1,
\quad\;
\sigma(T^*)\\\0=\sigma_{\kern-1ptP}(T^*)\\\0,
\quad\hbox{and}\;\quad
0\in\sigma(T^*).
$$
Since ${\#\sigma_{\kern-1ptP}(T^*)\le1}$, either
$\sigma_{\kern-1ptP}(T^*)=\{\lambda\}$ for ${\lambda\ne0}$, or
${\sigma_{\kern-1ptP}(T^*)\sse\0}$.

\vskip6pt\noi
(a) Suppose ${\sigma_{\kern-1ptP}(T^*)=\{\lambda\}}$ for some
${0\ne\lambda\in\CC}.$ Since ${0\in\sigma(T^*)}$ and
${\sigma(T^*)\\\0}={\sigma_{\kern-1ptP}(T^*)\\\0}$ we get
$\sigma(T^*)=\{{0,\lambda}\}.$ If $\X$ is a Banach space, then $\sigma(T)$ is
a compact nonempty set such that $\sigma(T)=\sigma(T^*)=\{{0,\lambda}\}$ (see,
e.g., \cite[Proposition VII.6.1]{Con} --- for Hilbert-space adjoint this
becomes ${\sigma(T)=\sigma(T^*)^*\!=\{{0,\overline\lambda}\}}$, which does
not alter the next argument)$.$ The spectrum of a weakly l-sequentially
supercyclic opera\-tor $T$ on a Banach space is such that all components of
the spectrum meet one and the same circle about the origin of the complex
plane for some finite (nonnegative) radius$.$ (Again, this has been verified
for supercyclic operators on a Hilbert space in \cite[Proposition 3.1]{Her},
and for weakly hypercyclic operators on a Banach space regarding the unit
circle in \cite[Theorem 3]{DT}, and extended to weakly supercyclic operators
on a Banach space in \cite[Proposition 3.5]{BM})$.$ Then $\{\lambda\}\ne\0$
cannot be a component of $\sigma(T)$, and hence
$\sigma_{\kern-1ptP}(T^*)\ne\{\lambda\}$ for $\lambda\ne0$, leading to a
contradiction$.$

\vskip6pt\noi
(b) Thus $\sigma_{\kern-1ptP}(T^*)\sse\0$ so that $\sigma(T^*)=\sigma(T)=\0$
and $T$ is quasinilpotent.
\end{proof}

\vskip0pt\noi
\begin{remark}

(a) A hypercyclic operator is not compact, and there is no supercyclic
operator on a complex normed space with finite dimension greater than 1
\cite[\hbox{Section 4}]{Hez}$.$ There are, however, compact supercyclic
operators on a separable infinite-dimensional complex Banach space
\cite[Theorem 1 and Section 4]{Hez}.

\vskip4pt\noi
(b) The Volterra operator is an example of a compact quasinilpotent that is
not supercyclic (and so it is not weakly l-sequentially supercyclic)
but, according to item (a) above and Theorems 4.1, 4.2, there exist
quasinilpotent supercyclic operators$.$ It was also show in
\cite[Corollary 5.3]{Sal} that if the adjoint of a bilateral or of a
unilateral weighted shift on $\ell^2$ or on $\ell_+^2$ has a weighting
sequence possessing a subsequence that goes to zero, then there is an
infinite-dimensional subspace whose all nonzero vectors are supercyclic for
it$.$ In particular this happens for weighted shifts with weighting sequences
converging to zero, and so this happens for the adjoint of compact
quasinilpotent weighted shifts.
\end{remark}

\vskip-10pt\noi
\bibliographystyle{amsplain}

\end{document}